\documentclass[12pt]{article}
\usepackage{amsmath,amssymb,amsthm}
\usepackage{graphicx}
\usepackage{subfigure}
\usepackage{amsmath,amsfonts,amssymb,amscd}
\usepackage{indentfirst,graphicx,epstopdf}
\usepackage{graphicx,psfrag,ifpdf,enumerate}
\usepackage{caption}
\usepackage{amsthm, amsfonts, color}
\usepackage[bookmarksnumbered, plainpages]{hyperref}
\usepackage[numbers,sort&compress]{natbib}
\usepackage[noend]{algpseudocode}
\usepackage{algorithmicx,algorithm}
\usepackage{psfrag}
\usepackage{color}
\usepackage{enumerate}
\usepackage{epstopdf}
\input{epsf}

\newcommand\tabcaption{\def\@captype{table}\caption}

\newtheorem{theorem}{Theorem}[section]
\newtheorem{conj}[theorem]{Conjecture}
\newtheorem{lemma}[theorem]{Lemma}
\newtheorem{claim}{Claim}

\theoremstyle{definition}

\title{\Large \bf Note on rainbow cycles in edge-colored graphs}
\date{}
\author{{\small Xiaozheng Chen, Xueliang Li}\\
{\small Center for Combinatorics and LPMC}\\
{\small Nankai University, Tianjin 300071, China}\\
{\small Email: cxz@mail.nankai.edu.cn, lxl@nankai.edu.cn}
}

\begin{document}

\maketitle

\begin{abstract}
Let $G$ be a graph of order $n$ with an edge-coloring $c$,
and let $\delta^c(G)$ denote the minimum color degree of $G$.
A subgraph $F$ of $G$ is called rainbow if all edges of $F$
have pairwise distinct colors. There have been a lot results on
rainbow cycles of edge-colored graphs. In this paper, we show that
(i) if $\delta^c(G)>\frac{3n-3}{4}$,
then every vertex of $G$ is contained in a rainbow triangle;
(ii) $\delta^c(G)>\frac{3n}{4}$,
then every vertex of $G$ is contained in a rainbow $C_4$;
and (iii) if $G$ is complete, $n\geq 8k-18$ and $\delta^c(G)>\frac{n-1}{2}+k$,
then $G$ contains a rainbow cycle of length at least $k$. Some gaps in
previous publications are also found and corrected. \\[3mm]
{\bf AMS Classification 2020:} 05C15, 05C38.

\end{abstract}

\section{Introduction}

We consider finite simple undirected graphs.
An edge-coloring of a graph $G$ is a mapping $c: E(G)\rightarrow \mathbb{N}$,
where $\mathbb{N}$ is the set of natural numbers.
A graph $G$ is called an \emph{edge-colored graph}
if $G$ is assigned an edge-coloring.
The color of an edge $e$ of $G$ and the set of colors assigned to $E(G)$
are denoted by $c(e)$ and $C(G)$, respectively.
For $V_1,V_2\subset V(G)$ and $V_1\cap V_2=\emptyset$,
we set $E(V_1,V_2)=\{xy\in E(G),x\in V_1,y\in V_2\}$,
and when $V_1=\{v\}$,
we write $E(u,V_2)$ for $E(\{u\},V_2)$.
The set of colors appearing on the edges between $V_1$ and $V_2$ in $G$
is denoted by $C(V_1,V_2)$.
When $V_1=\{v\}$, use $C(v,V_2)$ instead of $C(\{v\},V_2)$.
For a subgraph $T$ of $G$, the set of colors appearing on $E(T)$ is denoted by $C(T)$,
and use $T^C$ to denote $G-T$.
A subset $F$ of edges of $G$ is called \emph{rainbow}
if no distinct edges in $F$ receive the same color,
and a graph is called \emph{rainbow} if its edge-set is rainbow.
Specially, a path $P$ is \emph{rainbow} if no distinct edges in $E(G)$
are assigned the same color.
The length of a path $P=u_1u_2\cdots u_p$ is the number of edges in $E(P)$,
denoted by $\ell(P)$.
We use $u_iPu_j$ to denote the segment between $u_i$ and $u_j$ on $P$.
If $i<j$, then $u_iPu_j=u_iu_{i+1}\cdots u_{j}$;
if $i>j$, then $u_iPu_j=u_iu_{i-1}\cdots u_j$.
For a vertex $v\in V(G)$, the \emph{color-degree}
of $v$ in $G$ is the number of distinct colors assigned to the edges incident to $v$,
denoted by $d^c_G(v)$.
We use $\delta^c(G)=\mbox{min}\{d^c_G(v):v\in V(G)\}$
to denote the \emph{minimum color-degree} of $G$.
The set of neighbors of a vertex $v$
in a graph $G$ is denoted by $N_G(v)$.
Let $N^c(v)$ be a subset of $N_G(v)$
such that $|N^c(v)|=d^c_G(v)$
and each color in $C(v,N_G(v))$
appears in $E(v,N^c(v))$ exactly once.
For each vertex $v\in V(G)$ and a color subset $C'=\{c_1,c_2,\cdots,c_k\}$ of $C(G)$,
let $N_{C'}(v)=\{u~|~u\in N^c(v),c(uv)\in C'\}$.
Let $S$ be a subset of $V$,
we denote $N_{C'}(v)\cap S$ by $N_{C'}(v,S)$.
When $S=V(P)$,
we use $N_{C'}(v,P)$ instead of $N_{C'}(v,V(P))$.
For other notation and terminology not defined here,
we refer to \cite{B}.

The existence of rainbow substructures in edge-colored graphs
has been widely studied in literature. We mention here only those known results that are related to our paper.
For short rainbow cycles, a minimum color-degree condition for the existence of
a rainbow triangle was given by Li \cite{L} in 2013.

\begin{theorem}[\cite{L}]\label{th11}
Let $G$ be an edge-colored graph of order $n\geq 3$.
If $\delta^c(G)>\frac{n}{2}$,
then $G$ has a rainbow triangle.
\end{theorem}

In 2014, Li et al. \cite{LX} improved Theorem \ref{th10}
and got the following result.

\begin{theorem}[\cite{LX}]\label{th10}
Let $G$ be an edge-colored graph of order $n\geq 3$
satisfying one of the following conditions:

(1) $\sum_{u\in V(G)}d^c(u)\geq \frac{n(n+1)}{2}$,

(2) $\delta^c(G)\geq \frac{n}{2}$ and $G\notin \{K_{\frac{n}{2},\frac{n}{2}},K_4,K_4-e\}$.\\
Then $G$ contains a rainbow triangle.
\end{theorem}

There is also a result about the rainbow triangles in an edge-colored complete graph
which was proved by Fujita and Magnant in 2011.

\begin{theorem}[\cite{FM}]\label{th12}
Let $K_n$ be an edge-colored complete graph of order $n\geq 3$.
If $\delta^c(K_n)\geq \frac{n+1}{2}$,
then each vertex of $K_n$ is contained in a rainbow triangle.
\end{theorem}

In this paper, we get the the following result.

\begin{theorem}\label{th21}
Let $G$ be an edge-colored graph of order $n\geq 3$.
If $\delta^c(G)>\frac{3n-3}{4}$,
then each vertex of $G$ is contained in a rainbow triangle.
\end{theorem}

Next, for the rainbow $C_4$, \v{C}ada et al. in \cite{CKRY} obtained the following result.

\begin{theorem}[\cite{CKRY}]\label{th20}
Let $G$ be an edge-colored graph of order $n$.
If $G$ is triangle-free and $\delta^c(G)>\frac{n}{3}+1$,
then $G$ contains a rainbow $C_4$.
\end{theorem}

In this paper, we also get a result as follows.

\begin{theorem}\label{th22}
Let $G$ be an edge-colored graph of order $n\geq 3$.
If for each vertex $v$ of $G$, $d^c(v)>\frac{3n}{4}$,
then each vertex of $G$ is contained in a rainbow $C_4$.
\end{theorem}

Finally, for long rainbow cycles,
Li and Wang in \cite{LW} got the following result.

\begin{theorem}[\cite{LW}]\label{th13}
Let $G$ be an edge-colored graph of order $n\geq 8$.
If for each vertex $v$ of $G$,
$d^c(v)\geq d\geq \frac{3n}{4}+1$,
then $G$ has a rainbow cycle of length at least $d-\frac{3n}{4}+2$.
\end{theorem}

In 2016, \v{C}ada et al. in \cite{CKRY}
obtained a result on rainbow cycles of length at least four.

\begin{theorem}[\cite{CKRY}]\label{th14}
Let $G$ be an edge-colored graph of order $n$.
If for each vertex $v$ of $G$, $d^c(v)>\frac{n}{2}+2$,
then $G$ contains a rainbow cycle of length at least four.
\end{theorem}

At the end of their paper \cite{CKRY}, they raised the following conjecture.

\begin{conj}[\cite{CKRY}]\label{conj1}
Let $G$ be an edge-colored graph of order $n$ and $k$ be a positive integer.
If for each vertex $v$ of $G$, $d^c(v)>\frac{n+k}{2}$,
then $G$ contains a rainbow cycle of length at least $k$.
\end{conj}

Inspired by Theorem \ref{th14}, Tangjai in \cite{W} proved the following result.

\begin{theorem}[\cite{W}]\label{1}
Let $G$ be an edge-colored graph of order $n$ and $k$ be a positive integer.
If $G$ has no rainbow cycle of length 4 and $\delta^c(G)\geq \frac{n+3k-2}{2}$,
then $G$ contains a rainbow cycle of length at least $k$,
where $k\geq 5$.
\end{theorem}

However, we found some gaps in the proofs of Theorems \ref{th14} and \ref{1},
for which we will give our corrections in Section 3. We will show the following
result.

\begin{theorem}\label{0}
Let $G$ be an edge colored complete graph of order $n$
and $k$ be a positive integer larger than 5.
If $n\geq 8k-18$ and $\delta^c(G)> \frac{n-1}{2}+k$,
then $G$ contains a rainbow cycle of length at least $k$.
\end{theorem}

In order to prove our main result, we need the following result of Li and Chen
in \cite{CL} on the existence of long rainbow paths.

\begin{theorem}[\cite{CL}]\label{th7}
Let $G$ be an edge-colored graph, where $\delta^c(G)\geq t\geq 7$. Then
the maximum length of rainbow paths in $G$ is at least $\lceil\frac{2t}{3}\rceil+1$.
\end{theorem}

In the following sections, we will give the proofs of our three results, Theorems \ref{th21}, \ref{th22}
and \ref{0}, as well as the corrections of Theorems \ref{th14} and \ref{1}.

\section{Proofs of Theorems \ref{th21} and \ref{th22} }

To present the proof of Theorems \ref{th21} and \ref{th22},
we need some auxiliary lemmas.
In an edge-colored graph $G$,
a subset $A$ of $V$ is said to have the \emph{dependence property}
with respect to a vertex $v\notin A$ (denoted $DP_v$)
if $c(aa')\in \{c(va),c(va')\}$ for all $aa'\in E(G[A])$.

\begin{lemma}\label{lem21}
If a subset $A$ of vertices of an edge-colored graph $G$ has $DP_v$,
then there exists a vertex $x\in A$ such that
the number of colors of edges incident with $x$ in $G[A]$ which
are different from $c(vx)$ is at most $\frac{|A|-1}{2}$.
\end{lemma}

\begin{proof}
We prove this lemma by an oriented graph.
Orient the edges in $E(G[A])$ by applying the following rule:
for an edge $xy$,
if $c(xy)=c(vx)$, then the orientation of $xy$ is from $y$
to $x$;
otherwise, the orientation of $xy$ is from $x$ to $y$.
Thus, we get an oriented graph $D(A)$.
Apparently, the arcs with colors different from $c(vx)$
are out-arcs from $x$.
Let $x_0$ be a vertex in $D(A)$ with minimum out-degree.
Clearly, $d^+_{D(A)}(x_0)\leq \frac{|A|-1}{2}$.
Thus, the number of colors of edges incident with $x_0$ in $G[A]$
which are different from $c(vx_0)$ is at most $\frac{|A|-1}{2}$.

\end{proof}

Let $T$ be a subgraph of $G$ and
$C(G)\setminus C(T)$ be the set of colors not appearing on $T$.
Set $N_{C(G)\setminus C(T)}(v)=\{u~|~u\in N^c(v), c(uv)\notin C(T)\}$
and $S=N_{C(G)\setminus C(T)}(v)\cap N_{C(G)\setminus C(T)}(u)$ for $uv\in E(T)$.
Then we give the following lemma.

\begin{lemma}\label{lem22}
Let $G$ be an edge-colored graph of order $n\geq 3$
and $T$ be a $K_2$ or a rainbow triangle in $G$.
Then $S\neq \emptyset$
and
$$ |S|\geq\left\{
\begin{array}{rcl}
2\delta^c(G)-n-3,       &      & \mbox{$T$ is $C_3$}\\
2\delta^c(G)-n,     &      &  \mbox{$T$ is $K_2$}.
\end{array} \right. $$
\end{lemma}

\begin{proof}
Since $T$ is a $K_2$ or a rainbow triangle,
we have $S\in V(T^C)$.
Clearly, $|S|\geq |N_{C(T^C)}(v)|+|N_{C(T^C))}(u)|-(n-|T|)$.
Since
$$ |N_{C(T^C)}(v)|\geq\left\{
\begin{array}{rcl}
\delta^c(G)-3,       &      & \mbox{$T$ is $C_3$}\\
\delta^c(G)-1,     &      &  \mbox{$T$ is $K_2$},
\end{array} \right. $$\\
and $\delta^c(G)>\frac{n}{2}$,
we can get the result.

\end{proof}

\textbf{Proof of Theorem \ref{th21}}:
Let $G$ be a graph satisfying the assumptions of Theorem \ref{th21} and
suppose, to the contrary, that there exists a vertex $v$ such that no rainbow triangle
contains it.
Let $u\in N^c(v)$,
$T=G[\{u,v\}]$
and $S=N_{C(G)\setminus C(T)}(v)\cap N_{C(G)\setminus C(T)}(u)$.
According to Lemma \ref{lem22},
$S\neq \emptyset$.
For any edge $e=xy\in E(G[S])$, we have $c(xy)\in\{c(vx),c(vy)\}$;
otherwise $xyv$ is a rainbow triangle.
Clearly, $S$ has the dependence property with respect to $v$.
According to Lemma \ref{lem21},
there is a vertex $x_0\in S$ such that
the number of colors of edges incident with $x_0$ in $G[S]$
which are different from $c(vx_0)$ is at most $\frac{|S|-1}{2}$.
Since $c(vx_0)=c(ux_0)$ (otherwise $uvx_0$ is a rainbow triangle),
we have $|N^c(x_0)\cap (S\cup \{u,v\})|\leq \frac{|S|+1}{2}$.
Thus, $N^c(x_0)\cap (V(G)\setminus (S\cup \{u,v\}))\geq \delta^c(G)-\frac{|S|+1}{2}$.
So we have $n-|S|-2\geq \delta^c(G)-\frac{|S|+1}{2}$.
Then $2n-2\delta^c(G)-3\geq |S|$.
As $|S| \geq 2\delta^c(G)-n$ from Lemma \ref{lem22}, we have
$\delta^c(G)\leq \frac{3n-3}{4}$.
This contradiction completes the proof.

$\hfill\qedsymbol$

\textbf{Proof of Theorem \ref{th22}}:
Let $G$ be a graph satisfying the assumptions of Theorem \ref{th22} and
suppose, to the contrary, that there exists a vertex $v$ such that no rainbow $C_4$
contains it.
From Theorem \ref{th21}, there is a rainbow triangle $uvw$.
Let $S=N_{C(G)\setminus C(T)}(v)\cap N_{C(G)\setminus C(T)}(u)$.
According to Lemma \ref{lem22},
$S\neq \emptyset$.
For any vertex $x\in S$,
$c(vx)=c(ux)$;
since else $xuwvx$ is a rainbow $C_4$.
For any edge $e=xy\in E(G[S])$,
$c(vx)\neq c(uy)$;
else, since $c(vx)\neq c(vy)$, $c(vx)\neq c(ux)$, a contradiction.
Then $c(xy)\in \{c(ux),c(vy),c(uv)\}$;
otherwise $uvxy$ is a rainbow $C_4$.
We construct a new edge-colored subgraph $G'[S]$
by removing all the edges that have color $c(uv)$ in $E(G[S])$.
Then for any edge $e=xy\in E(G'[S])$,
$c(xy)\in \{c(ux),c(vy)\}=\{c(vx),c(vy)\}$.
Clearly, $S$ has the dependence property with respect to $v$ in $G'[S]$.
According to Lemma \ref{lem21},
there is a vertex $x_0$ in $S$ such that
the number of colors of edges incident with $x_0$ in $G[S]$
which are different from $c(vx_0)$ and $c(uv)$
is at most $\frac{|S|-1}{2}$.
If $w\in N^c(x_0)$
and $c(wx_0)\notin\{c(wu),c(uv),c(vx_0)\}$,
$wuvx_0w$ is a rainbow $C_4$;
if $c(wx_0)=c(wu)$, $wx_0vuw$ is a rainbow $C_4$.
Thus $c(wx_0)\in \{c(uv),c(vx_0)\}$.
Note that $c(u x_0)=c(v x_0)$.
Then $|N^c(x_0)\cap (S\cup \{u,v,w\})|\leq \frac{|S|-1}{2}+2$.
Thus, $N^c(x_0)\cap (V(G)\setminus (S\cup \{u,v,w\}))\geq \delta^c(G)-\frac{|S|-1}{2}-2$.
So, we have $n-|S|-3\geq \delta^c(G)-\frac{|S|-1}{2}-2$.
Then, $2n-2\delta^c(G)\geq |S|$.
As $|S| \geq 2\delta^c(G)-n-3$ from Lemma \ref{lem22}, we have
$\delta^c(G)\leq \frac{3n}{4}$.
This contradiction completes the proof.

$\hfill\qedsymbol$

\section{Proof of Theorem \ref{0}}

To present the proof of Theorem \ref{0},
we need some auxiliary theorems and lemmas.
Lemmas \ref{lem10} and \ref{lem12} are used to prove Theorem \ref{th14}.
We will use them to prove our theorem.

\begin{lemma}[\cite{CKRY}]\label{lem10}
Let $G$ be an edge-colored graph of order $n$
and $P=u_1u_2\cdots u_p$ be a longest rainbow path in $G$.
If $G$ contains no rainbow cycle of length at least $k$,
where $k\leq p$, then for any color $a\in C(u_1,u_kPu_p)$
and vertex $u_i\in V(u_kPu_p)$, where $c(u_1u_i)=a$,
there is an edge $e\in E(u_1Pu_i)$ such that $c(e)=a$.
\end{lemma}

Similarly, we have the following lemma.
\begin{lemma}\label{lem11}
Let $G$ be an edge-colored graph of order $n$
and $P=u_1u_2\cdots u_p$ be a longest rainbow path in $G$.
If $G$ contains no rainbow cycle of length at least $k$,
where $k\leq p$, then for any positive integers $s,t$ with $t\geq s+(k-1)$,
we have $c(u_su_t)\in C(u_sPu_t)$.
\end{lemma}

\begin{lemma}[\cite{CKRY}]\label{lem12}
Let $G$ be an edge-colored graph of order $n$
and $P=u_1u_2\cdots u_p$ be a longest rainbow path in $G$.
If $G$ contains no rainbow cycle of length at least $k$,
where $k\leq p$, then for any positive integers $s,t$ such that $s+t=k$,
we have $|C(u_1,u_kPu_{p-(t-1)})\cap C(u_p,u_sPu_{p-(k-1)})|\leq 1$.
\end{lemma}

\begin{lemma}\label{lem13}
Let $G$ be an edge-colored complete graph of order $n$
and $P=u_1u_2\cdots u_p$ be a longest rainbow path in $G$.
If $G$ contains no rainbow cycle of length at least $k$,
where $2k-1\leq p$, then $c(u_1u_p)=C(u_1,u_kPu_{p-(k-2)})$ or
$c(u_1u_p)=C(u_p,u_{k-2}Pu_{p-(k-1)})$.
\end{lemma}

\begin{proof}
From Lemma \ref{lem10},
we have $c(u_1u_k)\in C(u_1Pu_k)$ and $c(u_pu_{p-(k-1)})\in C(u_{p-(k-1)}Pu_p)$.
Then we have $c(u_1u_p)=c(u_1u_k)$ or $c(u_1u_p)\in C(u_{k}Pu_p)$;
otherwise $u_1u_kPu_pu_1$ is a rainbow cycle of length at least $k$.
Similarly, we have $c(u_1u_p)=c(u_pu_{p-(k-1)})$ or $c(u_1u_p)\in C(u_1Pu_{p-(k-1)})$.
Since $p\geq 2k-1$, we have $c(u_1u_k)\neq c(u_pu_{p-(k-1)})$.
Thus, $c(u_1u_p)=c(u_1u_k)\in C(u_1Pu_k)$
or $c(u_1u_p)=c(u_pu_{p-(k-1)})\in C(u_{p-(k-1)}Pu_p)$.
If $c(u_1u_p)=c(u_1u_k)$,
for any vertex $u_i\in V(u_kPu_{p-(k-2)})$,
we have $c(u_1u_i)=c(u_1u_p)$;
otherwise $u_1u_iPu_pu_1$ is a rainbow cycle of length at least $k$.
By the symmetry,
if $c(u_1u_p)=c(u_pu_{p-(k-1)})$,
we have $c(u_pu_i)=c(u_1u_p)$ for $u_i\in V(u_{k-2}Pu_{p-(k-1)})$.

\end{proof}

\textbf{Proof of Theorem \ref{0}}:
Let $G$ be a graph satisfying the assumptions of Theorem \ref{0}.
Suppose to the contrary, that $G$ contains no rainbow cycle of length at least $k$.
Let $P=u_1u_2\cdots u_p$ be a longest rainbow path in $G$.
Since $n\geq 8k-18$,
from Theorem \ref{th7} it follows that $p\geq 3k-5$.
From Lemma \ref{lem13}, w.l.o.g., suppose that $c(u_1u_p)=C(u_1,u_kPu_{p-(k-2)})$.

\vspace{2mm}
Set

$\begin{array}{lll}
A_1=C(u_1,u_kPu_{p-1}),&A_2=C(u_1,u_2Pu_{k-1}),&\\
B_1=C(u_p,u_{k-2}Pu_{p-(k-1)}),& B_2=C(u_p,u_{2}Pu_{k-3}),& B_3=C(u_p, u_{p-(k-2)}Pu_{p-1}),
\end{array} $\\
and

$\begin{array}{l}
C_0=(C(u_1, P^C)\setminus C(u_1,P))\cap (C(u_p, P^C)\setminus C(u_p,P)),
\end{array} $

$\begin{array}{ll}
C_1=C(u_1,P^C)\setminus (C_0\cup c(u_1,P)),&C_2=C(u_p,P^C)\setminus (C_0\cup c(u_p,P)).
\end{array} $

\begin{claim}\label{c2}
Let $u_s,u_t\in V(P)$ with $\ell(u_sPu_t)\geq 2k-3$.
Then for any pair of vertices $u_a$ and $u_b$,
if $k-1 \leq \ell(u_aPu_b)\leq \ell(u_sPu_t)-(k-2)$ and $c(u_su_t)\in C(u_aPu_b)$,
we have $c(u_au_b)=c(u_su_t)$.
\end{claim}

\begin{proof}
Since $\ell(u_aPu_b)\geq k-1$, from Lemma \ref{lem11}
we have $c(u_au_b)\in C(u_aPu_b)$.
If $c(u_au_b)\neq c(u_su_t)$,
then $u_sPu_au_bPu_tu_s$ is a rainbow cycle of length at least $k$,
see Figure \ref{fig1}, a contradiction.

\end{proof}

\begin{figure}[htbp]
  \centering
 \scalebox{1}{\includegraphics[width=2.0in,height=0.5in]{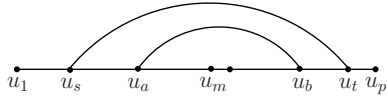}}\\
 \captionsetup{font={scriptsize}}
  \caption{For Claim 1}\label{fig1}
\end{figure}

\begin{claim}\label{c1}
$|A_1\cap B_1|\leq 1$.
\end{claim}

\begin{proof}
From Lemma \ref{lem12},
setting $t=2$ and $s=k-2$, the claim follows obviously.

\end{proof}

\begin{claim}\label{c3}
$|A_1\cap B_2\setminus (B_1\cup \{c(u_1u_p)\})|\leq 1$.
\end{claim}

\begin{proof}
Suppose that $c(u_1u_t)=c(u_pu_s)=m\in A_1\cap B_2\setminus (B_1\cup \{c(u_1u_p)\}) $.
Then $t\in [p-(k-3),p-1]$ and $s\in [2,k-3]$.
Since $m\notin\{ c(u_pu_{s+(k-2)}),c(u_1u_p)\}$,
we have $m\in C(u_sPu_{s+(k-2)})$
and $m\in C(u_{t-(k-2)}Pu_t)$;
otherwise $u_sPu_{s+(k-2)}u_pu_s$
and $u_1u_{t-(k-2)}Pu_tu_1$ are rainbow cycles of length $k$.
Thus, $m\in C(u_{t-(k-2)}Pu_{s+{k-2}})\subseteq C(u_{k-2}Pu_{p-(k-2)})$.
Since $p\geq 3k-5$,
from Lemma \ref{lem11} and Claim \ref{c2},
$l(u_{k-2}Pu_{p-(k-2)})\geq k-1$.
Hence, $c(u_{k-2}u_{p-(k-2)})=m$,
see Figure \ref{fig2}.

Suppose, to contrary, there is another color $m'\in A_1\cap B_2\setminus (B_1\cup \{c(u_1u_p)\})$
such that $m'\neq m$.
Then $c(u_{k-2}u_{p-(k-2)})=m'$, a contradiction.

\begin{figure}[htbp]
  \centering
 \scalebox{1}{\includegraphics[width=2.5in,height=0.8in]{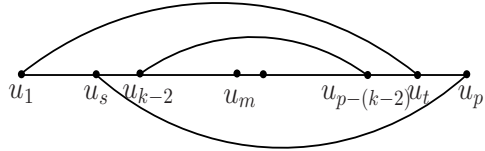}}\\
 \captionsetup{font={scriptsize}}
  \caption{For the proof of Claim 3}\label{fig2}
\end{figure}

\end{proof}

\begin{claim}\label{c6}
Let $D=\{x\in V(P^C): x\in N_{C_1\cup C_0}(u_1,P^C)\cap N_{C_2\cup C_0}(u_p,P^C) ~and~ c(u_1x)\neq c(u_px)\}$.
Then $|D|\leq 2$.
\end{claim}

\begin{proof}
Suppose to the contrary, that $|D|\geq 3$.
According to Lemma \ref{lem13},
we have $c(u_1u_{p-(k-2)})=c(u_1u_p)$, and then $c(u_1u_{p-(k-2)})\notin C(u_1,D)\cup C(u_p,D)$.
Since $|D|\geq 2$,
there exists a vertex $x\in D$
such that $c(u_1u_x)\neq c(u_pu_{k-1})$.
Then, $c(u_pu_{k-1}) \notin \{c(u_1x),c(u_p x)\}$.
Thus, we have $c(u_1 x)\in C(u_2Pu_{k-1})$ and $c(u_p x)\in C(u_{p-(k-2)}Pu_{p-1})$
or $c(u_1 x)\in C(u_{p-(k-2)}Pu_{p})$ and $C(u_p x)\in c(u_{1}Pu_{k-2})$;
otherwise, $u_1u_{p-(k-2)}Pu_p xu_1$ or $u_1Pu_{k-1}u_p xu_1$ is a rainbow cycle of length at least $k$, see Figure \ref{fig7}.
Since $|D|\geq 3$,
there exists a vertex $x_0\in D$ such that
$c(u_1x_0)\notin\{c(u_pu_{k-1}), c(u_pu_{p-(k-1)})\}$.
Then, $c(u_1x_0)\in C(u_1Pu_{k-1})\cup C(u_{p-(k-2)}Pu_p)$.
Thus, $u_1u_kPu_{p-(k-1)}u_px_0u_1$
is a rainbow cycle of length at least $k$, see Figure \ref{fig6}, a contradiction.
The proof is thus complete.
\begin{figure}[htbp]
  \centering
 \scalebox{1}{\includegraphics[width=3.5in,height=0.8in]{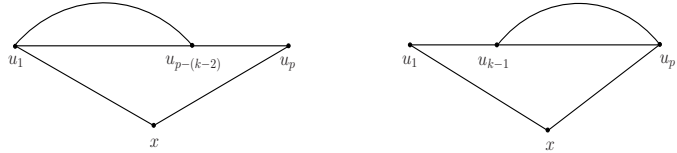}}\\
 \captionsetup{font={scriptsize}}
  \caption{Otherwise in the proof of Claim 4}\label{fig7}
\end{figure}

\begin{figure}[htbp]
  \centering
 \scalebox{1}{\includegraphics[width=2.5in,height=0.8in]{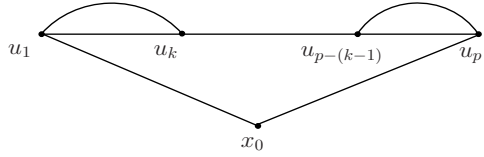}}\\
 \captionsetup{font={scriptsize}}
  \caption{A contradiction in the proof of Claim 4}\label{fig6}
\end{figure}

\end{proof}

It is easy to see that $c(u_1u_2)\notin B_1\cup B_2$
and $c(u_pu_{p-1})\notin A_1$.
However, possibly $c(u_1u_2)\in A_1$
or $c(u_pu_{p-1})\in B_1\cup B_2$.
Thus, we set
\vspace{2mm}

$\varepsilon_1 =
\begin{cases}
    1 & \mbox{if $c(u_1u_2)\notin A_1$}, \\
    0 & \mbox{if $c(u_1u_2)\in A_1$};
\end{cases}$

$\varepsilon_2 =
\begin{cases}
    1 & \mbox{if $c(u_pu_{p-1})\notin B_1\cup B_2$}, \\
    0 & \mbox{if $c(u_pu_{p-1})\in B_1\cup B_2$};
\end{cases}$
\vspace{2mm}

According to Lemma \ref{lem13},
we know that $c(u_1u_p)\in A_1$ and $c(u_1u_p)\in C(u_1Pu_k)$.
Thus, we set

\vspace{2mm}
$\varepsilon_3 =
\begin{cases}
    1 & \mbox{if $c(u_1u_p)\notin B_1$}, \\
    0 & \mbox{if $c(u_1u_{p})\in B_1$};
\end{cases}$

\vspace{2mm}
Then,
\begin{equation}\label{e0}
\begin{array}{ll}
|A_1|+|C_0|+|C_1|+\varepsilon_1&\geq d^c(u_1)-|A_2\setminus \{c(u_1u_2)\})|\geq \delta^c(G)-(k-3),
\end{array}
\end{equation}
\begin{equation}\label{e1}
\begin{array}{ll}
|B_1|+|B_2\setminus (B_1\cup \{c(u_1u_p)\})|
+|C_0|+|C_2|+\varepsilon_2+\varepsilon_3&\\
\geq d^c(u_p)-|B_3\setminus \{c(u_{p}u_{p-1})\})|\geq \delta^c(G)-(k-3).&
\end{array}
\end{equation}

By Claim \ref{c6},
we can get $|N_{C_1}(u_1,P^C)\cap N_{C_2}(u_p,P^C)|\leq 2$.
Furthermore, if there is a vertex $x\in N_{C_2}(u_p,P^C)$
such that $c(u_1x)\in C_0$.
Then there is another distinct vertex $y$
such that $c(u_py)=c(u_1x)$.
If $c(u_1y)\in C_1$,
then there is no other vertex $z\in N_{C_2}(u_p,P^C)$ such that
$c(u_1z)\in C_0$
or $z\in N_{C_1}(u_1,P^C)$ and $c(u_pz)\in C_0$.
Otherwise, $|D|\geq 3$, a contradiction.
Hence, there are at least $|C_0|-1$ vertices in
$\{x\in V(P^C)~|~c(u_1x)=c(u_px)\in C_0\}$.
\begin{equation}\label{e2}
\begin{array}{lll}
|V(P^C)|&\geq & |N_{C_1}(u_1,P^C)\cup N_{C_2}(u_p,P^C)|+|\{x\in V(P^C)~|~c(u_1x)=c(u_px)\in C_0\}|\\
&\geq& |C_1|+|C_2|+|C_0|-2-1.
\end{array}
\end{equation}

Note that for any color $a\in C_0$,
there is an edge in $P$ whose color is $a$;
since else $P$ is not a longest rainbow path.
Then we have
\begin{equation}\label{e3}
\begin{array}{lll}
|V(P)|&=&|E(P)|+1\\
& \geq &|e\in E(P), c(e)\in A_1\cup B_1\cup B_2|+|e\in E(p), c(e)\in C_0|\\
& & +\varepsilon_1+\varepsilon_2+1\\
&\geq &|A_1|+|B_1|+|B_2\setminus (B_1\cup c(u_1u_p))|+\varepsilon_1+\varepsilon_2-1.
\end{array}
\end{equation}

Summarizing (\ref{e2}) and (\ref{e3}), we have
\begin{equation}\label{e4}
\begin{array}{lll}
n&\geq&|V(P)|+|V(P^C)|\\
&\geq &(|A_1|+|C_0|+ |C_1|+\varepsilon_1)+(|B_1|+|B_2\setminus (B_1\cup c(u_1u_p))|+|C_0|+|C_2|+\varepsilon_2)-4,
\end{array}
\end{equation}
Combining with inequalities (\ref{e0}) and (\ref{e1}),
we get that
\begin{equation}\label{e5}
\begin{array}{lll}
n&\geq & 2(\delta^c(G)-(k-3))-\varepsilon_3-4\\
&\geq & 2\delta^c(G)-2k+1.
\end{array}
\end{equation}
Then,
$\delta^c(G)\leq \frac{n-1}{2}+k$,
a contradiction.
The proof is thus complete.

$\hfill\qedsymbol$\\

In the rest of the paper, we analyze the mistake in the proof of Theorem \ref{th14}.
In paper \cite{CKRY},
the authors thought that $|V(P^C)|\geq |C_0|+|C_1|+|C_2|-1$,
where $C_i$, $i=0,1,2$ are of the same definition as in our proof of Theorem \ref{0}.
Note that they proved that
$|N_{C_1}(u_1,P^C)\cap N_{C_2}(u_p,P^C)|\leq 1$
under the condition of rainbow $C_4$-free.
Then the inequality implies that if for any color $a\in C_0$,
there is a vertex $x\in V(P^C)$ such that $c(u_1x)=a$,
one would have $x\notin N_{C_2}(u_p,P^C)$;
or $c(u_px)=a$,
and $x\notin N_{C_1}(u_p,P^C)$.
Actually, we can not prove that the assertion is correct.
However, assume that there are two vertices
$x_i\in V(P^C)$ such that $c(u_1x_i)=a_i\in C_0$ and
$c(u_px_i)=c_i\in C_2$, where $i=1,2$ and $a_1\neq a_2$.
Then there is another vertex $y\in V(P^C)$
such that
$c(u_py)=a_1\in C_0$.
If $c(u_1y)=c_0\in C_1$,
then note that $c_0\neq c_2$.
Then $u_1yu_px_2u_1$ is a rainbow $C_4$, see Figure \ref{fig5},
a contradiction.
Hence, there are at least $|C_0|-1$ vertices in
$\{x\in V(P^C)~|~c(u_1x)\in C_0 \mbox{~and~}c(u_px)\in C_0\}$.
Therefore, $|V(P^C)|\geq |C_0|+|C_1|+|C_2|-2$.
\begin{figure}[htbp]
  \centering
 \scalebox{1}{\includegraphics[width=2.0in,height=1.0in]{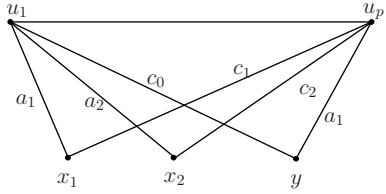}}\\
 \captionsetup{font={scriptsize}}
  \caption{ A rainbow $C_4$}\label{fig5}
\end{figure}

Thus, Theorems \ref{th14} and \ref{1}
can be corrected by the following statements.

\begin{theorem}
Let $G$ be an edge-colored graph of order $n$.
If for every vertex $v$ of $G$, $d^c(v)>\frac{n+5}{2}$,
then $G$ contains a rainbow cycle of length at least four.
\end{theorem}

\begin{theorem}
Let $G$ be an edge-colored graph of order $n$ and $k$ be a positive integer.
If $G$ has no rainbow cycle of length 4 and $\delta^c(G)> \frac{n+3k-3}{2}$,
then $G$ contains a rainbow cycle of length at least $k$,
where $k\geq 5$.
\end{theorem}

\end{document}